\documentclass{adamjoucc}

\usepackage{hyperref}
\usepackage{amssymb}
\usepackage[capitalize]{cleveref}
\usepackage{relsize}

\DeclareMathOperator{\Cay}{Cay}
\DeclareMathOperator{\PSL}{PSL}
\DeclareMathOperator{\Stab}{Stab}

\newcommand{\NN}{\mathbb{N}}
\newcommand{\ZZ}{\mathbb{Z}}

\newcommand{\1}{\mathbf{1}}
\newcommand{\cartprod}{\mathbin{\raise0.1pt\hbox{\smaller[2]\hskip0.1pt$\square$\hskip0.1pt}}}
\newcommand{\caut}{\mathcal{C}}
\newcommand{\compose}{\circ}
\newcommand{\inner}[1]{\widehat{#1}}
\newcommand{\iso}{\cong}
\newcommand{\ltrans}{\mathcal{L}}

\newcommand{\pref}[1]{(\ref{#1})}

\newcommand{\fullcref}[2]{\cref{#1}\pref{#1-#2}}

\newcommand{\csee}[1]{(see \cref{#1})}
\newcommand{\fullcsee}[2]{(see \fullcref{#1}{#2})}

\newcommand{\MR}[1]{\href{https://mathscinet.ams.org/mathscinet-getitem?mr=#1}{MR\,#1}}
\newcommand{\doi}[1]{\href{https://doi.org/#1}{doi:#1}}

\theoremstyle{plain} 
\newtheorem{cor}[equation]{Corollary}
\newtheorem{lem}[equation]{Lemma}
\newtheorem{prop}[equation]{Proposition}
\newtheorem{thm}[equation]{Theorem}

\Crefname{cor}{Corollary}{Corollaries}
\Crefname{lem}{Lemma}{Lemmas}
\Crefname{prop}{Proposition}{Propositions}
\Crefname{thm}{Theorem}{Theorems}

\theoremstyle{definition}
\newtheorem*{ack}{Acknowledgments}
\newtheorem{defn}[equation]{Definition}
\newtheorem{eg}[equation]{Example}
\newtheorem{notation}[equation]{Notation}
\newtheorem{rem}[equation]{Remark}
\newtheorem{rems}[equation]{Remarks}

\Crefname{defn}{Definition}{Definitions}

\AddToHook{env/cor/begin}{\crefalias{equation}{cor}}
\AddToHook{env/lem/begin}{\crefalias{equation}{lem}}
\AddToHook{env/prop/begin}{\crefalias{equation}{prop}}
\AddToHook{env/thm/begin}{\crefalias{equation}{thm}}
\AddToHook{env/defn/begin}{\crefalias{equation}{defn}}
\AddToHook{env/eg/begin}{\crefalias{equation}{eg}}
\AddToHook{env/notation/begin}{\crefalias{equation}{notation}}
\AddToHook{env/rem/begin}{\crefalias{equation}{rem}}
\AddToHook{env/rems/begin}{\crefalias{equation}{rems}}

\newcounter{case}

\renewcommand{\thecase}{\arabic{case}}
\crefformat{case}{Case~#2#1#3}
\Crefname{case}{Case}{Cases}
\crefname{case}{case}{cases}

\newcounter{step}

\renewcommand{\thestep}{\arabic{step}}
\crefformat{step}{Step~#2#1#3}
\Crefname{step}{Step}{Steps}
\crefname{step}{case}{cases}

\makeatletter
\newcommand{\noprelistbreak}{\smallskip\@nobreaktrue\nopagebreak} 
\makeatother

\makeatletter
\@finaltrue
\renewenvironment{frontmatter}
{\thispagestyle{plain}}
{\vskip 20pt%
\blfootnote{\raggedright \ifnum\@authorcount=1\textit{E-mail address:}\else\textit{E-mail addresses:}\fi ~\@emails}%
\if@proofsline\global\linenumbers\fi%
}

\renewenvironment{abstract}
{\hrule height 0.25pt
\vskip 5pt
\noindent \textbf{Abstract}
\vskip 5pt
}
{
\vskip 5pt
\noindent \textit{\small Keywords:~\@keywords}

\vskip 3pt
\noindent \textit{\small Math.\ Subj.\ Class. (2020): \@msc}
\vskip 5pt
\hrule height 0.25pt
}
\def\@oddrunninghead{Composing group automorphisms with colour-preserving automorphisms}
\def\@evenrunninghead{\@oddrunninghead}
\makeatother
\usepackage{enumitem}
    \setlist{itemsep=0pt}
	\setlist[enumerate, 1]{label =\textup{(\arabic*)}, ref=\arabic*, topsep=\smallskipamount}
	\setlist[itemize, 1]{topsep=\smallskipamount}
	\setlist[enumerate, 2]{label =\textup{(\alph*)}, ref=\theenumi\alph*, topsep=0pt}
	\setlist[itemize, 2]{label=$\circ$, topsep=0pt}

\newcommand{\defit}[1]{\emph{\bfseries #1}} 

\begin{document}

\begin{frontmatter}   

\titledata{Composing group automorphisms with
	\\  colour-preserving automorphisms 
	\\ of Cayley graphs}{}

\authordata{Dave Witte Morris}
{Department of Mathematics and Computer Science, 
University of Lethbridge, 
\\ 
4401 University Drive, 
Lethbridge, Alberta, T1K~3M4, Canada}
{dmorris@deductivepress.ca, https://deductivepress.ca/dmorris}
{}



\begin{abstract}
We show that if $\varphi$ is a colour-permuting automorphism of a connected, finite Cayley graph, and the order of the Cayley graph is either odd or square-free, then $\varphi$~is the composition of a group automorphism and a colour-preserving graph automorphism.
Some analogous results are also established for isomorphisms between two different Cayley graphs.
\end{abstract}

\end{frontmatter}

\section{Introduction}

\begin{defn}[cf.\ {\cite[Defns.~1.1 and~1.2]{HKMM}}]
Let $S$ be a subset of a group~$G$. If $S$ is \defit{symmetric} (which means $s^{-1} \in S$ for all $s \in S$), then the \defit{Cayley graph} $\Cay(G; S)$ is the (undirected) graph whose vertices are the elements of~$G$, with an edge joining $g$ and~$gs$ for all $g \in G$ and $s \in S$.
This graph has a natural edge-colouring: the edge joining $g$ and~$gs$ is coloured with the set $\{s^{\pm1}\}$.
	\begin{itemize}
	\item An automorphism $\varphi$ of $\Cay(G; S)$ is \defit{colour-preserving} if it preserves the colours of the edges. (In other words, for all $g \in G$ and $s \in S$, we have $\varphi(gs) \in \{\varphi(g) \, s^{\pm1}\}$.) 
	\item An isomorphism $\varphi$ from one Cayley graph $\Cay(G_1; S_1)$ to another Cayley graph $\Cay(G_2; S_2)$ is \defit{chromatic} if it respects the edge-colourings. More concretely, there is a bijection $\pi \colon S_1 \to S_2$, such that $\varphi(gs) \in \{\varphi(g) \, \pi(s)^{\pm1} \}$ for all $g \in G_1$ and $s \in S_1$.
	\end{itemize}
\end{defn}

It is easy to see that if $\alpha \colon G_1 \to G_2$ is a group isomorphism, then it is a chromatic isomorphism from $\Cay(G_1; S_1)$ to $\Cay \bigl( G_2; \alpha(S_1) \bigr)$, for any symmetric subset~$S_1$ of~$G_1$ (cf.\ \cite[p.~190]{HKMM}). It is also obvious that every colour-preserving automorphism is a chromatic isomorphism. The following result shows in the odd-order case that every chromatic isomorphism is a combination of these two obvious types.

\begin{thm} \label{OddChromIso}
If\/ $\Cay(G_1; S_1)$ and\/ $\Cay(G_2; S_2)$ are connected Cayley graphs of odd order, then every chromatic isomorphism from\/ $\Cay(G_1; S_1)$ to\/ $\Cay(G_2; S_2)$ is the composition of a colour-preserving automorphism of\/ $\Cay(G_1; S_1)$ and a group isomorphism from $G_1$ to~$G_2$.
\end{thm}

\begin{rems}
\leavevmode\noprelistbreak
	\begin{enumerate}
	\item The \lcnamecref{OddChromIso} was already known to be true, by a very different proof, when $G_1$ and~$G_2$ are abelian, without any assumption on the orders of the groups, or even that the groups are finite 
		\cite[Thm.~1.2]{AlimirzaeiMorris-abelian}. 
	\item By applying the \lcnamecref{OddChromIso} to the inverse of~$\varphi$, we can reverse the order of the composition: $\varphi$ is also the composition of a group isomorphism from $G_1$ to~$G_2$ and a colour-preserving automorphism of\/ $\Cay(G_2; S_2)$.
	\end{enumerate}
\end{rems}

The \lcnamecref{OddChromIso} implies that every group of odd order is determined by any of its (edge-coloured) connected Cayley graphs:

\begin{cor} \label{OddChromIsoImpliesGIso}
Assume $G_1$ is a finite group of odd order, and $G_2$ is any finite group. If there is a chromatic isomorphism from some connected Cayley graph of~$G_1$ to some Cayley graph of~$G_2$, then $G_1$ is isomorphic to~$G_2$.
\end{cor}

\begin{rem} \label{NotSameGroup}
It is well known that \cref{OddChromIsoImpliesGIso}'s restriction on the order of~$G_1$ cannot be omitted. (So it also cannot be omitted in \cref{OddChromIso}.)
For example, let 
	\[ D_{2n} = \langle \, a, b \mid a^n = b^2 = abab = 1 \,\rangle \]
be the dihedral group of order~$2n$. Then $\ZZ_n \oplus \ZZ_2 \not\iso D_{2n}$ (if $n \ge 3$), but there is a chromatic isomorphism
	\[ \text{from \ $\Cay \bigl( \ZZ_n \oplus \ZZ_2; \{ \pm(1,0), (0,1)\} \bigr)$ \  to \ $\Cay \bigl( D_{2n}; \{a^{\pm1}, b\} \bigr)$.} \]
\end{rem}

Although \cref{NotSameGroup} shows that \cref{OddChromIso} cannot be extended to include all groups of square-free order, the following result accomplishes this when $G_1 = G_2$.

\begin{thm} \label{SquareFreeAut}
If\/ $\Cay(G; S_1)$ and\/ $\Cay(G; S_2)$ are connected Cayley graphs of a group~$G$ of square-free order, then every chromatic isomorphism from\/ $\Cay(G; S_1)$ to\/ $\Cay(G; S_2)$ is the composition of a colour-preserving automorphism of\/ $\Cay(G; S_1)$ and an automorphism of the group~$G$.
\end{thm}

\begin{defn}[cf.\ {\cite[Defns.~1.2 and 1.4]{HKMM}}]
\leavevmode
\noprelistbreak
\begin{itemize}
	\item A bijection $\varphi$ from a group~$G_1$ to a group~$G_2$ is \defit{affine} if it is the composition of a left translation and a group isomorphism. More concretely, there exist an element $g_1 \in G_1$ and a group isomorphism $\tau \colon G_1 \to G_2$, such that $\varphi(x) = \tau(g_1 \, x)$.
	
	\item A chromatic isomorphism from a Cayley graph to itself is called a \defit{colour-permuting automorphism}.

	\item A Cayley graph is:
		\begin{itemize}
		\item \defit{CCA} if every colour-preserving automorphism is affine;
		\item \defit{strongly CCA} if every colour-  permuting automorphism is affine. 
		\end{itemize}
	\end{itemize}
\end{defn}

It is known that if every connected Cayley graph of a finite group~$G$ of odd order is CCA, then every connected Cayley graph of~$G$ is strongly CCA \cite[Prop.~6.4]{HKMM}. (The converse is obvious.) The implication ($\ref{CCAiffStrong-n} \Rightarrow \ref{CCAiffStrong-cca}$) of the following consequence of \cref{OddChromIso,SquareFreeAut} strengthens this by establishing that the equivalence holds for each individual Cayley graph, not only when all of them are CCA. It also generalizes the earlier result, by allowing the order to be square-free, rather than odd. The reverse implication shows that the generalization (plus two small exceptional cases) is best possible.

\begin{cor} \label{CCAiffStrong}
For $n \in \NN^+$, the following are equivalent:
	\begin{enumerate}
	\item \label{CCAiffStrong-cca}
	Every CCA Cayley graph of order~$n$ is strongly CCA.
	\item \label{CCAiffStrong-n}
	Either $n$ is odd, or $n$ is square-free, or $n \in \{4, 8\}$.
	\item \label{CCAiffStrong-colperm}
	Every colour-permuting automorphism of every connected Cayley graph of order~$n$ is the composition of a colour-preserving graph automorphism and a group automorphism.
	\end{enumerate}
\end{cor}

The above results apply when the order of the graph is either odd or square-free.  The following \lcnamecref{OddSquarefree} provides a much more precise statement when the order satisfies both conditions. The equivalence ($\ref{OddSquarefree-strong} \Leftrightarrow \ref{OddSquarefree-CCA}$) is immediate from \cref{CCAiffStrong}, and the equivalence ($\ref{OddSquarefree-CCA} \Leftrightarrow \ref{OddSquarefree-F21}$) was proved by E.\,Dobson, A.\,Hujdurovi\'c, K.\,Kutnar, and J.\,Morris \cite[Thm.~3.12]{DobsonEtAl-OddSquarefree}.

\begin{cor} \label{OddSquarefree}
If\/ $\Cay(G;S)$ is a connected Cayley graph of a finite group~$G$ whose order is odd and square-free, then the following are equivalent:
	\begin{enumerate}
	\item \label{OddSquarefree-strong}
	$\Cay(G;S)$ is \underline{not} strongly CCA.
	\item \label{OddSquarefree-CCA}
	$\Cay(G;S)$ is \underline{not} CCA.
	\item \label{OddSquarefree-F21}
	$G = F_{21} \times H$, where $F_{21}$ and~$H$ are subgroups of~$G$, such that 
		\begin{enumerate}
		\item $F_{21}$ is a nonabelian group of order\/~$21$, 
		\item $S \cap F_{21}$ consists of precisely four elements of order\/~$3$,
		and
		\item $S = (S \cap H) \cup (S \cap F_{21})$, so 
			\[ \Cay(G;S) \iso \Cay(F_{21}; S \cap F_{21}) \cartprod \Cay(H; S \cap H ) . \]
		\end{enumerate}
	\end{enumerate}
\end{cor}

\begin{ack}
This research was partially supported by a grant from the Natural Science and Engineering Research Council of Canada.
\end{ack}

\section{Proofs}

\begin{notation}
For $i = 1,2$, assume $G_i$ is a group, and $\Cay(G_i; S_i)$ is a Cayley graph of~$G_i$.%
\noprelistbreak
	\begin{enumerate}
	\item For each $g \in G_i$, we let $\ell^{G_i}_g \colon G_i \to G_i$ be the left-translation $x \mapsto gx$. This is a colour-preserving automorphism of $\Cay(G_i; S_i)$ \cite[p.~190]{HKMM}.
	\item For any subgroup~$H_i$ of~$G_i$, we let $\ltrans^{G_i}_{H_i} = \{\, \ell^{G_i}_h \mid h \in H_i \,\}$.
	\item If $\varphi$ is a chromatic isomorphism from $\Cay(G_1; S_1)$ to $\Cay(G_2; S_2)$, and $H_2$ is a subgroup of~$G_2$, we let
		\[(\ltrans^{G_2}_{H_2})^\varphi = \{\, \varphi^{-1} \compose \ell^{G_2}_h \compose \varphi \mid h \in H_2 \,\} . \]
	This is a group of colour-preserving automorphisms of $\Cay(G_1; S_1)$.
	\end{enumerate}
\end{notation}

\begin{defn}[cf.\ {\cite[p.~88]{Isaacs-GrpThy}}]
A subgroup~$H$ of a a finite group~$G$ is a \defit{2-complement} in~$G$ if $|H|$ is odd, and the index $|G : H|$ is a power of~$2$.
\end{defn}

\begin{eg} \label{Not4Comp}
If $|G|$ is not divisible by~$4$, then $G$ has a 2-complement.  Indeed, if $|G|$ is odd, then $G$ is a 2-complement in itself. For the remaining case, it is a standard exercise to show that if $|G| = 2k$, where $k$ is odd, then $G$ has a subgroup of order~$k$ \cite[Thm.~27.2, p.~347]{Gallian}.

Solvable groups always have a 2-complement \cite[Thm.~3.14, p.~87]{Isaacs-GrpThy}, but many other groups do not. For example, the only nonabelian finite simple groups that have a 2-complement are $\PSL(2,r)$, where $r$ is a Mersenne prime \cite[Thm.~1.3]{Simple2Comp}.
\end{eg}

For our purposes, it is important to know that the 2-complement is essentially unique, if it exists:

\begin{lem}[Arad-Ward {\cite[Cor.~4.5]{AradWard}}] \label{2CompConj}
All 2-complements in any finite group are conjugate.
\end{lem}

This has the following consequence, which is fundamental for the proofs of our main results.

\begin{prop} \label{AffFor2Comp}
Assume
	\begin{itemize}
	\item $\Cay(G_1; S_1)$ and\/ $\Cay(G_2; S_2)$ are connected Cayley graphs, 
	\item $\varphi$ is a chromatic isomorphism from\/ $\Cay(G_1; S_1)$ to\/ $\Cay(G_2; S_2)$,
	and
	\item $H_i$ is a 2-complement in~$G_i$, for $i = 1,2$.
	\end{itemize}
Then there is a colour-preserving automorphism $\psi$ of $\Cay(G_1; S_1)$, 
and a group isomorphism $\alpha \colon H_1 \to H_2$, such that if we let $\varphi' = \varphi \compose \psi$, then $\varphi'(h g) = \alpha(h) \, \varphi'(g)$ for all $h \in H_1$ and $g \in G_1$.
\end{prop}

\begin{proof}
We apply the Frattini argument (cf.\ \cite[proof of Lem.~1.13, p.~15]{Isaacs-GrpThy}).
For convenience, let $\ltrans_i = \ltrans^{G_i}_{H_i}$ for $i = 1,2$. Also let 
	\[ \text{$\caut$ be the group of colour-preserving automorphisms of $\Cay(G_1; S_1)$} , \]
so $\ltrans_1$ and $\ltrans_2^\varphi$ are subgroups of~$\caut$.

Since $\Cay(G_1; S_1)$ is connected, it is easy to see, for every vertex $v \in G_1$, that the stabilizer $\Stab_{\caut}(v)$ is a $2$-group \cite[Lem.~6.3]{HKMM}, so it is immediate from the Orbit-Stabilizer Theorem \cite[Thm.~C-1.16, p.~9]{Rotman-AdvAlg2} that $|\caut : \ltrans^{G_1}_{G_1}|$ is a power of~$2$. Since $|G_1 : H_1|$ is also a power of~$2$, we conclude that $|\caut : \ltrans_1|$ is a power of~$2$. So $\ltrans_1$ is a 2-complement in~$\caut$. 
Since $\ltrans_1 = |H_1| = |H_2| = |\ltrans_2|$, then $\ltrans_2^\varphi$ is also a 2-complement in~$\caut$. Therefore, \cref{2CompConj} tells us there is some $\psi \in \caut$, such that $\ltrans_2^{\varphi \compose \psi} = \ltrans_1$. 

For each $h \in H_1$, the conclusion of the preceding paragraph implies there is some $\alpha(h) \in H_2$, such that 
	\[ \ell_{\alpha(h)} \compose \varphi \compose \psi = \varphi \compose \psi \compose \ell_h . \]
In other words, for all $g \in G_1$, we have $\alpha(h) \, \varphi \bigl( \psi(g) \bigr) = \varphi \bigl( \psi( h \, g )$.

From this, it is not difficult to see that $\alpha$ respects multiplication. Also, $\alpha$ must be a bijection, since $\varphi \compose \psi$ is a bijection. Hence, $\alpha$ is a group isomorphism from~$H_1$ to~$H_2$.
\end{proof}

\begin{proof}[\bf Proof of \cref{OddChromIso}]
Assume that $\varphi$ is a chromatic isomorphism from $\Cay(G_1; S_1)$ to $\Cay(G_2; S_2)$. Since $|G_1|$ is odd, the group $G_1$ is a 2-complement in~$G_1$. Hence, \cref{AffFor2Comp} provides a colour-preserving automorphism $\psi$ of $\Cay(G_1; S_1)$, and a group isomorphism $\alpha \colon G_1 \to G_2$, such that 
	\[ \text{$(\varphi \compose \psi)(h g) = \alpha(h) \, (\varphi \compose \psi)(g)$ \ for all $g,h \in G_1$} . \]
By precomposing $\psi$ with an appropriate left-translation $\ell_a$ (and replacing $\alpha$ with the automorphism $h \mapsto \alpha(aha^{-1})$), we may assume $(\varphi \compose \psi)(\1_{G_1}) = \1_{G_2}$. Then letting $g = \1_{G_1}$ in the above equation yields 
	\[ (\varphi \compose \psi)(h) 
	= (\varphi \compose \psi)(h \cdot \1_{G_1})
	= \alpha(h) \, (\varphi \compose \psi)(\1_{G_1})
	= \alpha(h) \cdot 1_{G_2}
	= \alpha(h) . \]
So $\varphi = \alpha \compose \psi^{-1}$  is the composition of the colour-preserving graph automorphism~$\psi^{-1}$ and the group isomorphism~$\alpha$.
\end{proof}

We will use the following simple observation:

\begin{lem}[cf.\ {\cite[Lem.~2.2(4a)]{AlimirzaeiMorris-abelian}}] \label{|phi(s)|} 
If $\varphi$ is a chromatic isomorphism from $\Cay(G_1; S_1)$ to $\Cay(G_2; S_2)$, such that $\varphi(\1_{G_1}) = \1_{G_2}$, then $|\varphi(s)| = |s|$ for all $s \in S_1$.
\end{lem}

\begin{proof}
We have $\varphi(\1_{G_1}) = \1_{G_2}$ by assumption, so $\varphi$ sends edge $\bigl( \1_{G_1}, s \bigr)$ of $\Cay(G_1; S_1)$ to edge $\bigl( \1_{G_2}, \varphi(s) \bigr)$ of $\Cay(G_2; S_2)$. Since $\varphi$ is a chromatic isomorphism, this implies that $\varphi$ maps all of the edges of colour $\{s^{\pm1}\}$ in $\Cay( G_1; S_1)$ to edges of colour $\{\varphi(s)^{\pm1}\}$ in $\Cay( G_2; S_2)$. Hence, the map $\varphi$~induces an isomorphism from $\Cay \bigl( G_1; \{s^{\pm1}\} \bigr)$ to $\Cay \bigl( G_2; \{\varphi(s)^{\pm1}\} \bigr)$. Each component of the first graph has $|s|$ vertices, and each component of the second graph has $|\varphi(s)|$ vertices, so we must have $|s| = |\varphi(s)|$.
\end{proof}

The terminology in part~\pref{SuperSolvDefn-solv} of the following \lcnamecref{SuperSolvDefn} is motivated by the fact that a group is said to be ``supersolvable'' if it satisfies the condition, except that each subgroup $H_i$ is only required to be normal, rather than characteristic \cite[Problem 3B.7, p.~85]{Isaacs-GrpThy}. The notion in part~\pref{SuperSolvDefn-meta} was classically known as ``metacyclic\rlap,'' but modern authors use this term for the weaker condition that $G$ has a cyclic normal subgroup~$N$, such that $G/N$ is cyclic \cite{MetacyclicHistory}. Therefore, we add the adjective ``strongly'' to avoid conflict with the terminology that is currently in use.

\begin{defn} \label{SuperSolvDefn}
Let $H$ be a group.
\noprelistbreak
	\begin{enumerate}
 	\item \label{SuperSolvDefn-char}
	Recall that a subgroup~$N$ of~$H$ is \defit{characteristic} if $\alpha(N) = N$ for every automorphism~$\alpha$ of~$H$ \cite[p.~11]{Isaacs-GrpThy}.
	\item \label{SuperSolvDefn-solv}
	We will say that~$H$ is \defit{characteristically supersolvable} if there is a series 
	\[ \{\1\} = H_0 \subseteq H_1 \subseteq \cdots \subseteq H_r = H , \]
such that each $H_k$ is a characteristic subgroup of~$H$, and the quotient $H_k/H_{k-1}$ is cyclic for $k = 1,2,\ldots, r$.
	\item \label{SuperSolvDefn-meta}
	$H$ is \defit{strongly metacyclic} if its commutator subgroup $[H,H]$ and its abelianization $H/[H,H]$ are cyclic.
	\end{enumerate}
\end{defn}

\begin{eg} \label{SuperEg}
Well-known examples of characteristically supersolvable groups~$H$ include:
	\begin{enumerate}
	\item cyclic groups ($r = 1$),
	\item strongly metacyclic groups ($r = 2$ and $H_1 = [H,H]$),
	and
	\item \label{SuperEg-squarefree}
	groups of square-free order (because it is well known that they are strongly metacyclic \cite[Thm.~V.3.11, p.~175]{Zassenhaus-ThyGrps}).
	\end{enumerate}
\end{eg}

We will see that (the even-order case of) \cref{SquareFreeAut} is a special case of the following result.

\begin{thm} \label{SuperSolv}
If $G$ has a characteristically supersolvable 2-complement~$H$, such that $|G : H| = 2$, then every colour-permuting automorphism of every connected Cayley graph of~$G$ is the composition of a colour-preserving graph automorphism and a group automorphism. Therefore, every CCA Cayley graph of~$G$ is strongly CCA.
\end{thm}

\begin{proof}[Proof\/ \normalfont (similar to the proof of ($2 \Rightarrow 1$) of {\cite[Cor.~6.13]{HKMM})}]
Let $\varphi$ be a chromatic isomorphism from $\Cay(G; S_1)$ to $\Cay(G; S_2)$. We see from \cref{AffFor2Comp} that, by precomposing $\varphi$ with a colour-preserving automorphism, we may assume there is an automorphism $\alpha$ of~$H$, such that
	\begin{align} \label{SuperSolvPf-phi(hg)}
	\text{$\varphi(h g) = \alpha(h) \, \varphi(g)$ \quad for all $h \in H$ and $g \in G$} 
	. \end{align}
By further precomposing with a left-translation (which is another a colour-preserving automorphism), we may also assume $\varphi(\1) = \1$. (Since $|G:H| = 2$, we know that $H$ is a normal subgroup of~$G$, so the condition~\pref{SuperSolvPf-phi(hg)} will still hold, possibly with a different choice of~$\alpha$.) Also note that 
	\begin{align} \label{SuperSolvPf-phi(h)=alpha(h)}
	 \text{$\varphi(h) = \varphi(h \cdot \1) = \alpha(h) \, \varphi(\1) = \alpha(h) \cdot \1 = \alpha(h)$ \quad for all $h \in H$} 
	 . \end{align}

For $i = 1,2$, let 
	\[ S_i^* = \{\, t  \mid \text{$t$ is a nontrivial element of prime-power order in $\langle s \rangle$, for some $s \in S_i$} \,\} . \]
(Note that, since $S_i$ generates~$G$, the set $S_i^*$ also generates~$G$.) It is not difficult to see (by the argument in the proof of \cite[Lem.~6.11]{HKMM}) that $\varphi$ is a chromatic isomorphism from $\Cay(G; S_1^*)$ to $\Cay(G; S_2^*)$. Since we also have $\varphi(\1) = \1$, this implies
	\begin{align} \label{SuperSolvPf-phi(gt)in}
	\text{$\varphi(gt) \in \{\varphi(g) \, \varphi(t)^{\pm1}\}$ \ for all $g \in G$ and $t \in S_1^*$} 
	. \end{align}

To complete the proof, it will suffice to show that $\varphi$ is an automorphism of the group~$G$. To do this,
	\[ \text{we will show $\varphi(gt) = \varphi(g) \, \varphi(t)$ for all $g \in G$ and $t \in S_1^*$} . \]
To this end, first note that if $|t| = 2$, then $|\varphi(t)| = 2$ (by \cref{|phi(s)|}), so $\varphi(t) = \varphi(t)^{-1}$, and therefore the desired equality holds, by~\pref{SuperSolvPf-phi(gt)in}.  Also note that if $g \in H$, then (by~\pref{SuperSolvPf-phi(hg)} and~\pref{SuperSolvPf-phi(h)=alpha(h)}) we have
	\[ \varphi(gt) = \alpha(g) \, \varphi(t) = \varphi(g) \, \varphi(t) , \]
as desired.

Hence, we may assume $|t| \neq 2$ and $g \notin H$. 
Since $|t| \neq 2$, we know $|t|$ is odd (because it is a prime-power that divides $|G|$, and is not~$2$), so $t \in H$.
Also, since $g \notin H$, we know $|g|$ is even. Then $|\varphi(g)|$ is also even (by \cref{|phi(s)|}), so $\varphi(g) \notin H$. Since $|G : H| = 2$, this implies 
	\[ \text{there is some $h \in H$, such that $\varphi(g) = gh$} . \]
For $i \in 1,2$, and for each $a \in G_i$, let $\inner a \colon G_i \to G_i$ be the conjugation by~$a$, so 
	\[ \inner a(x) = a x a^{-1} . \]
To reduce the number of parentheses in the remainder of the argument, we will write $a \, b (t)$ for $a\bigl( b(t) \bigr)$. Then we have
	\[ \varphi(gt)
		= \varphi \bigl( \inner g(t) \cdot g \bigr)
		=  \alpha \, \inner g(t)  \cdot \varphi(g) 
		=  \varphi(g) \cdot \inner{\varphi(g)^{-1}} \,  \alpha  \, \inner g(t)  
		=  \varphi(g) \cdot \inner{h^{-1}}\,  \inner{g^{-1}} \, \alpha \, \inner g(t)  
		. \]
Now, we use the assumption that $H$ is characteristically supersolvable: there is a series
	\[ \{\1\} = H_0 \subseteq H_1 \subseteq \cdots \subseteq H_r = H  \]
of characteristic subgroups of~$H$, such that
	\[ \text{$H_k/H_{k-1}$ is cyclic for $k = 1,2,\ldots, r$} . \]
Note that, for each~$k \in \{1,\ldots,r\}$, the subgroups $H_k$ and~$H_{k-1}$ are invariant under every automorphism of~$H$, so  every automorphism of~$H$ induces an automorphism of $H_k/H_{k-1}$.

Let $k \in \{1,\ldots,r\}$ be minimal, such that $t \in H_k$. (So $t H_{k-1}$ is a nontrivial element of $H_k/H_{k-1}$. Then $\alpha(t) H_{k-1}$ is also a nontrivial element of $H_k/H_{k-1}$, because $\alpha$ induces an automorphism of the quotient.) Since $H_k/H_{k-1}$ is cyclic, its automorphism group is abelian, so
	\[ \inner{h^{-1}} \, \inner{g^{-1}} \, \alpha \, \inner g(t)   \equiv \inner{h^{-1}} \, \alpha(t) \pmod{H_{k-1}} . \]
Also note that, since $|h|$ is odd and $\alpha(t)$ represents a nontrivial element of $H_k/H_{k-1}$, we have
	\[ \inner{h^{-1}} \, \alpha(t) \not\equiv \alpha(t)^{-1} \pmod{H_{k-1}} . \]
Furthermore, recall that $\alpha(t) = \varphi(t)$ (by~\pref{SuperSolvPf-phi(h)=alpha(h)}). Putting all of this together, we have
	\[ \varphi(gt) \not\equiv \varphi(g) \, \varphi(t)^{-1} \pmod{H_{k-1}} , \]
so $\varphi(gt) \neq \varphi(g) \, \varphi(t)^{-1}$. We therefore conclude from~\pref{SuperSolvPf-phi(gt)in} that $\varphi(gt) = \varphi(g) \, \varphi(t)$.
\end{proof}

\begin{proof}[\bf Proof of \cref{SquareFreeAut}]
Since $|G|$ is square-free, it is not divisible by~$4$, so we know that $G$ has a 2-complement~$H$ \csee{Not4Comp}.
We may assume $|G|$ is even, for otherwise \cref{OddChromIso} applies; therefore $|G : H| = 2$. Now, since $|H|$ is a divisor of~$|G|$, we know that it is square-free, so $H$ is characteristically supersolvable \fullcsee{SuperEg}{squarefree}, so \cref{SuperSolv} applies.
\end{proof}

We will use the following two known results, but only in the special case where the order of the group is~$4$ or~$8$.

\begin{lem}[{\cite[Prop.~4.1]{HKMM}}] \label{abelian}
If $G$ is abelian, then every CCA Cayley graph of~$G$ is strongly CCA.
\end{lem}

\begin{lem}[cf.\ {\cite[Prop.~5.6]{HKMM}}] \label{dihedral}
If $G$ is a dihedral group, and $|G| \not\equiv 2 \pmod{4}$, then every connected Cayley graph of~$G$ is strongly CCA.
\end{lem}

\begin{proof}[\bf Proof of \cref{CCAiffStrong}]
($\ref{CCAiffStrong-colperm} \Rightarrow \ref{CCAiffStrong-cca}$)
This is easy.
Let $\varphi$ be a colour-permuting automorphism of a connected Cayley graph of order~$n$ that is CCA. By assumption, $\varphi$~is the composition of a colour-preserving automorphism and a group isomorphism. However, since the Cayley graph is CCA, every colour-preserving automorphism is affine. Therefore, $\varphi$ is the composition of two affine maps, and is therefore affine.
Since $\varphi$ is an arbitrary colour-permuting automorphism, we conclude that the Cayley graph is strongly CCA.

\medbreak

($\ref{CCAiffStrong-n} \Rightarrow \ref{CCAiffStrong-colperm}$)
The desired conclusion is immediate from \cref{OddChromIso,SquareFreeAut} if $n$ is either odd or square-free. Also, it is easy to see that every Cayley graph of order~$4$ is strongly CCA (or apply \cref{abelian}).

So we only need to consider the case where $n = 8$: let $\Cay(G; S)$ be a connected Cayley graph of order~$8$ that is CCA. If $G$ is either abelian or dihedral, then we know from \cref{abelian,dihedral} that $\Cay(G; S)$ is strongly CCA.  

So we may assume that $G$ is the quaternion group of order~$8$.  Then the automorphism group of~$G$ acts as the full symmetric group on the three colours of edges coming from the elements $\{\pm i, \pm j, \pm k\}$ of order~$4$ (and there is only one other edge-colour $\{-1\}$, which must therefore be fixed by any colour-permuting automorphism), so any colour-permuting automorphism of $\Cay(G; S)$ can be composed with a group automorphism to become colour-preserving.

\medbreak

($\ref{CCAiffStrong-cca} \Rightarrow \ref{CCAiffStrong-n}$)
We prove the contrapositive: assume $n$ is even, $n$ is not square-free, and $n \notin \{4, 8\}$. 

Suppose, for the moment, that $|G|$ is divisible by~$4$. Let $G = D_{n/2} \times \ZZ_2$, and let $S = \{a^{\pm1},b,c\}$, where $\{a,b\}$ is the natural generating set of $D_{n/2}$ (from \cref{NotSameGroup}) and $c$ is a generator of the $\ZZ_2$ factor. Then $\Cay(G; S)$ is chromatically isomorphic to the Cartesian product $C_{n/4} \cartprod P_2 \cartprod P_2$, where .$C_{n/4}$ is coloured $\{a^{\pm1}\}$, and the two $P_2$~factors are coloured $\{b\}$ and~$\{c\}$, respectively.
It is easy to see that this is CCA, because any colour-preserving automorphism must preserve each factor of the Cartesian product.
However, it is not strongly CCA, because there is a colour-permuting automorphism that interchanges the two $P_2$~factors (and hence interchanges the colours $\{b\}$ and~$\{c\}$), but the element~$c$ is in the centre of~$G$, whereas $b$ is not, so no group automorphism can interchange $b$ and~$c$.

We may now assume $n/2$ is odd. Then, since $n$ is not square-free, it must be divisible by the square of some odd prime~$p$. So we can write $n = 2 p^{\ell+1} m$, where $\gcd(2p, m) = 1$ and $\ell \ge 1$. Let 
	\[ G = D_{2p} \times (\ZZ_p)^\ell \times \ZZ_m , \]
and let $S$ be the generating set of~$G$ that combines a generator of each cyclic factor with the natural generating set $\{a,b\}$ of~$D_{2p}$ that appears in \cref{NotSameGroup}. Then $\Cay(G; S^{\pm1})$ is chromatically isomorphic to the Cartesian product $P_2 \cartprod (C_p)^{\ell+1} \cartprod C_m$, where $P_2$ is coloured $\{b\}$, the first $C_p$~factor is coloured~$\{a^|pm1\}$, and the other $C_p$~factors each have their own colour. As in the preceding paragraph, it is easy to see that this is CCA (because any colour-preserving automorphism must preserve each factor of the Cartesian product). On the other hand, it is not strongly CCA, because a colour-permuting automorphism can permute the factors of the Cartesian product, but each element of order~$p$ in a $\ZZ_p$ factor of~$G$ is central, whereas the element~$a$ of order~$p$ in $D_{2p}$ is not central, so no group automorphism can interchange the two corresponding $C_p$ factors of the Cartesian product.
\end{proof}

\end{document}